\documentclass[reqno,11pt]{amsart}
\usepackage{latexsym}
\usepackage{amsfonts}
\usepackage{amssymb}
\usepackage{mathrsfs}
\usepackage[all]{xy}
\usepackage{bbm}
\usepackage{xcolor}
\usepackage{ulem}
\usepackage[mathscr]{euscript}

\usepackage[english]{babel}
\usepackage{amsmath}
\usepackage{graphicx}
\usepackage[colorinlistoftodos]{todonotes}
\usepackage[colorlinks=true, allcolors=blue]{hyperref}
\usepackage{tikz-cd}
\usepackage{tikz}
\usepackage{pst-node}
\usepackage{mathtools}

\usepackage[T2A]{fontenc}

\theoremstyle{plain}
\newtheorem{thm}{Theorem}[section]

\newtheorem{rem}[thm]{Remark}
\newtheorem{prop}[thm]{Proposition}
\newtheorem{cor}[thm]{Corollary}
\newtheorem{defn}[thm]{Definition}

\theoremstyle{definition}

\theoremstyle{remark}

\numberwithin{equation}{section}
\newcommand{\BA}{\mathbf{BA}}
\newcommand{\Sing}{\mathbf{Sing}}
\newcommand{\DI}{\mathbf{DI}}
\newcommand{\SL}{\operatorname{SL}}

\newcommand{\diag}{\operatorname{diag}}

\newcommand\mr{M_{m,n}}

\newcommand\amr{$A\in M_{m,n}$}
\newcommand\da{Diophantine approximation}

\newcommand{\R}{{\mathbb{R}}}
\newcommand{\Z}{{\mathbb{Z}}}
\newcommand{\N}{{\mathbb{N}}}

\newcommand\hd{Hausdorff dimension}
\newcommand{\va}{{\boldsymbol{\alpha}}}
\newcommand{\vb}{{\boldsymbol{\beta}}}
\newcommand{\vw}{{\boldsymbol{\omega}}}

\newcommand{\ve}{{\bf e}}
\newcommand{\vp}{{\bf p}}
\newcommand{\vq}{{\bf q}}
\newcommand{\x}{{\bf x}}
\newcommand{\vx}{{\bf x}}
\newcommand{\vy}{{\bf y}}
\newcommand{\ignore}[1]{{}}

\newcommand\eq[2]{
\begin{equation}
\label{eq:#1}
{#2}
\end{equation}
}
\newcommand{\equ}[1]{\eqref{eq:#1}}

\title[Weighted Approximation]{Weighted Uniform Diophantine Approximation \\ of systems of  linear forms}
\author{Dmitry Kleinbock and Anurag Rao}
\address{Brandeis University, Waltham MA
02454-9110 {\tt kleinboc@brandeis.edu}}
\address{Wesleyan University, Middletown CT 06459-0260
{\tt arao@wesleyan.edu}}

\begin{document}

\begin{abstract}
Following the development of weighted asymptotic approximation properties of matrices,
% as in \cite{K98}, 
we introduce the analogous uniform approximation properties {(that is, study the improvability of Dirichlet's Theorem)}.
An added feature is the use of general norms, rather than the supremum norm, to quantify the approximation.
In terms of homogeneous dynamics, the approximation properties of an $m \times n$ matrix are governed by a trajectory in $\SL_{m+n}(\R)/\SL_{m+n}(\Z)$ avoiding %a collection of 
a compact {subset} of the space of lattices called the critical locus {defined with respect to the corresponding norm}.
The trajectory is formed by the action of a one-parameter diagonal subgroup corresponding to the weights. 
{We first state a very precise form of Dirichlet's theorem and prove it for some norms.
Secondly we show, for these same norms, that the set of Dirichlet-improvable matrices has full Hausdorff dimension.
Though the techniques used vary greatly depending on the chosen norm, we expect these results to hold in general.}
\end{abstract}

%\subjclass{11J04; 11J13, 37A17, 37D40}
\thanks{D.K.\ was  supported by  NSF grant  DMS-1900560.}
\date{February, 2022}

\subjclass[2010]{11J13; 11J83, 11H06, 37A17}

\maketitle
%\large

\section{Introduction}\label{intro}
Let $m$ and $n$ be positive integers and let $d=m+n$.  We will denote by  $\mr$ the space of $m\times n$ real matrices, and by $\|\cdot \|_\infty$ the supremum norm on $\R^m$, $\R^n$ and $\R^d$. 
The classical theorem of Dirichlet, see e.g.\ \cite[\S I.1.5]{Cassels}, 
asserts that for any \amr\ and $t>1$ there exists $(\vp,\vq)\in \Z^m\times (\Z^n\smallsetminus\{\mathbf{0}\})$ satisfying %the following system of inequalities:
 \begin{equation}\label{dirichletoriginal}
    \|A\vq -\vp \|^m_\infty \le 1/t\quad \textrm{and}\quad \|\vq\|^n_\infty < t.
    \end{equation}
  Here $A$ is viewed as a system of $m$ linear forms $A_1,\dots,A_m$ (rows of $A$) in $n$ variables, and the goal is to approximate the values of these forms at integer points by integers.  A natural question to ask is whether one can improve \eqref{dirichletoriginal} by replacing $1/t$ with a smaller function, that is, consider the following system of inequalities:
 \begin{equation}\label{dirichletpsi}
    \|A\vq -\vp \|^m_\infty < \psi(t)\quad \textrm{and}\quad  \|\vq\|^n_\infty <t,
  \end{equation}
where $\psi$ is a positive  function such that $\psi(t)$ is  strictly less than $\psi_1(t) := 1/t$ for all large enough $t$.
%Following the definition in 
One says that %an $m$ by $n$ real matrix 
$A$ is \textit{$\psi$-Dirichlet} (see \cite{kw1, kw, KSY}) if the system %of inequalities 
\eqref{dirichletpsi}
%\begin{align}\label{equ:dirichlet}
%\|A\bm{q}-\bm{p}\|^m<\psi(t)\quad \textrm{and}\quad \|\bm{q}\|^n<t
%\end{align}
has solutions in $(\vp,\vq)\in \Z^m\times (\Z^n\smallsetminus \{\mathbf{0}\})$ for all sufficiently large $t$. We will denote the set of $\psi$-Dirichlet matrices by $D_{\infty}(\psi)$. 
(The use of the subscript $\infty$ in \eqref{equ:DI} and in other occurrences below refers to the use of the supremum norm in \eqref{dirichletpsi}.)

The above set-up is usually referred to as \textit{uniform approximation}, as opposed to \textit{asymptotic
approximation} dealing with the system %of inequalities 
\eqref{dirichletpsi}
being solvable   for an unbounded set of $t$. Note that from Dirichlet's Theorem it trivially  follows that $D_{\infty}(c\psi_1) = \mr$ if $c > 1$, and with a little more work, caused by the difference between `$<$' in \eqref{dirichletoriginal}  and `$\le$' in \eqref{dirichletpsi}, one can show that $D_{\infty}(\psi_1) = \mr$ as well, see Theorem~\ref{perfect-Dirichlet-sup} below for a more general statement.

The problem of improving Dirichlet's theorem was initiated by Davenport and Schmidt  \cite{Davenport-Schmidt}  who showed that the set
 \begin{equation}\label{equ:DI}
\DI_{\infty} :=\bigcup_{0<c<1}D_\infty(c\psi_1)
  \end{equation}
of \textit{Dirichlet improvable} matrices
is of Lebesgue measure zero, while having full Hausdorff dimension $mn$. Furthermore, Davenport and Schmidt showed that $\DI $ contains the set  $\BA$ of  \textit{badly approximable} matrices
\begin{equation*}
 \BA := \left\{\text{\amr}:   \inf_{\vp \in \Z^m,\, \vq \in \Z^n\smallsetminus \{{\bf 0}\}} \|A\vq -\vp \|^m_\infty   \|\vq\|^n_\infty > 0 \right\},
\end{equation*}
which was known to be \textit{thick}, that is, have full \hd\ at any point of $\mr$ \cite{Schmidt-BA}.

%A recurring theme in \da\ is to generalize classical problems 
In this paper we will generalize the above set-up in several different ways. 
It is known that many results in \da\ extend to \textit{approximation with  weights}, an approach allowing to treat forms $A_i $ and components of $\vq$ differently. 
Namely, given a tuple of positive weights 
\begin{equation}\label{weights}
    \vw = (\va,\vb) \in \R_+^{m}\times \R_+^n \text{ with } \sum_{i=1}^{m} \alpha_i =   \sum_{i=1}^n \beta_i = 1,
\end{equation}
one introduces \textit{quasi-norms} associated with $\va$ and $\vb$ respectively:
\begin{align*}
\|\vx\|_{\va}:=\max_i|x_i|^{1/\alpha_i} \quad\textrm{and}\quad \|\vy\|_{\vb}:=\max_j|y_j|^{1/\beta_j}.
\end{align*}
Then, for $\psi$
%n approximation 
%function $\psi : \R_{>0} \to \R_{>0}$ 
as above, one  says that  \amr\ is $(\psi, \vw)$-\textit{Dirichlet}, denoted by $A\in D_{\infty,\vw}(\psi)$, if the system of inequalities 
\begin{equation*}\label{equ:dirichletwei}
\|A\vq-\vp\|_{\va}<\psi(t)\quad \textrm{and}\quad \|\vq\|_{\vb}<t
\end{equation*}
 in $(\vp,\vq)\in \Z^m\times (\Z^n\smallsetminus \{\mathbf{0}\})$ for all sufficiently large $t$.  % following the notation introduced in \cite[\S 2.1]{Kleinbock1998},
 In other words, we are considering the solvability of the system \begin{equation}\label{weightsystem}
  \begin{cases}  \left| A_i \cdot \vq - p_i\right| < \psi(t)^{\alpha_i}, &  i = 1,\dots,m;\\ \qquad \qquad \quad|q_j|<t^{\beta_j},\ &j = 1,\dots,n.\end{cases}
\end{equation}
%with solutions $\vp \in \Z^m$ and $\vq \in \Z^n \smallsetminus \{{\bf 0}\}$.
Clearly the unweighted case corresponds to the choice $$\va = (1/m,\dots,1/m)\text{ and }\vb  = (1/n,\dots,1/n).$$ 
 
 A lot of what can be proved for unweighted approximation easily extends to the weighted case. A weighted analogue of Dirichlet's theorem, which is a straightforward consequence of Minkowski's Convex Body Theorem \cite[\S III.2.2]{Cassels}, 
 implies that $D_\vw(c\psi_1) = \mr$ if $c > 1$.  
And with a little more work one can 
prove a stronger result:
%show that $D_\vw(\psi_1) = \mr$ as well.} Our first result is 

 \begin{thm}\label{perfect-Dirichlet-sup}
For %the supremum norm and 
any choice of weights $\vw$, we have $D_{\infty,\vw}(\psi_{1}) = \mr$.
\end{thm}

As for the set $$\DI_{\infty,\vw} :=\bigcup_{0<c<1}D_{\infty,\vw}(c\psi_1),$$ the fact that it has Lebesgue measure zero was established by the first named author and Weiss using the correspondence between \da\ and dynamics, see \cite[Theorem 1.4]{KW}. In this paper we prove

 \begin{thm}\label{thickness-Dirichlet-sup}
For %the supremum norm and 
any choice of weights $\vw$,  the set $\DI_{\infty,\vw}$ contains the set $\BA_\vw$ of $\vw$-badly approximable matrices, defined by
\begin{equation}\label{wBA} \BA_\vw := \left\{\text{\amr}:   \inf_{\vp \in \Z^m,\, \vq \in \Z^n\smallsetminus \{{\bf 0}\}} \|A\vq -\vp \|_\va  \|\vq\|_\vb > 0 \right\}.
\end{equation}
\end{thm}
Note that the latter set is thick, as shown  in \cite[\S 4.5]{KW}, see also
%See also the article of Pollington-Velani 
\cite{PV} and \cite{KW2}.
It should also be noted that in \cite[Theorem 4.6]{Suess}, Su\"ess proved the above result in the case when $m=1$.
Our proof here is different and is written in the language of dynamics on the space of lattices.

We remark that the problem of determining conditions on $\psi$ under which the set $D_{\infty,\vw}(\psi)$ has zero/full measure is rather tricky. A complete solution for the case $m=n=1$ is given in \cite{kw1}, and a   recent paper \cite{KSY} by the first named author, Strombergsson and Yu  deals with the general case, including arbitrary weights, and provides a partial result.

 In order to generalize the set-up further, let us restate the definition of $(\psi, \vw)$-{Dirichlet} matrices in a geometric language. 
Let $X_d$ denote the space of unimodular lattices in $\R^d$, identified with $\SL_d(\R)/\SL_d(\Z)$ via $g \mapsto g\Z^d$.
Given %a matrix
 \amr, we define
\begin{equation*}
 u_A := \left[ {\begin{array}{cc}
   I_m & A \\
   0 & I_n \\
  \end{array} } \right], \ \ \Lambda_A := u_A\Z^d.
  \end{equation*}
Then it is easy to see that $A\in D_\infty(\psi)$ if and only if \begin{equation}\label{geometric-dirichlet}
    \Lambda_A \cap  \left[ {\begin{array}{cc}
   \psi(t)^{1/m}I_m  & 0 \\
   0 & t^{1/n}I_n  \\
  \end{array} } \right]
  B_\infty(1) \ne \{{\bf 0}\}
\end{equation}
for all sufficiently large $t$ (here $B_\infty(1)$ is the unit open ball centered at zero with respect to the norm $\|\cdot\|_\infty$)
And for a weighted version it will be convenient to use the following notation for a number raised to a vector power: if $c >0$ and $\x\in\R^k$, define
\begin{equation*}
   c^\x := \diag(
    c^{x_1}, \dots , c^{x_k}).
  \end{equation*}
Then, similarly to \eqref{geometric-dirichlet}, one can state that $A\in D_{\infty, \vw}(\psi)$ if and only if \begin{equation}\label{geometric-dirichlet-weights}
    \Lambda_A \cap  \left[ {\begin{array}{cc}
   \psi(t)^{\va}  & 0 \\
   0 & t^{\vb}  \\
  \end{array} } \right]
  B_\infty(1) \ne \{{\bf 0}\}
\end{equation}
for all sufficiently large $t$.

At this point one might wonder: what will change if in the above definition the supremum norm $\|\cdot\|_\infty$ is replaced by some other norm $\nu$? and indeed this type of questions have appeared in the literature, first for the case $m=n=1$ \cite{AD}, and then for arbitrary $m,n$ in the unweighted case \cite{KR}.
% \comm{Was there supposed to be a citation to our paper here?} 
We will now use \eqref{geometric-dirichlet-weights} to state a general weighted definition. In order to do that, for an arbitrary norm $\nu$ on $\R^d$ let us define  
  the \textit{critical radius} of   $\nu$ as follows:
 $$
 r_\nu := \sup\big\{r:
 %\begin{aligned}\  
 \Lambda\cap B_\nu(r) = \{{\bf 0}\}%\\
  \text{ for some }\Lambda\in X_d%\ \end{aligned}
  \big\}. $$
  Here $B_\nu(r) := \{\x\in\R^d: \nu(\x) < r\}$; clearly $r_\infty = 1$. 
  (Throughout the paper we will use the notation $p$ when $\nu$ is the  $\ell^p$ norm, in particular when $p=\infty$.)

  %is the  open ball with respect to the norm $\nu$ of radius $r$ centered at zero with respect to the norm $\|\cdot\|_\infty$
  
Now let us define the most general sets of $\psi$-Dirichlet matrices.

\begin{defn}
Given a
%n approximation 
function $\psi : \R_{>0} \to \R_{>0}$ and a tuple of weights $\vw = (\va,\vb)$ as in \eqref{weights},
we say that \amr\ is $(\psi, \nu,\vw)$-\textit{Dirichlet} if 
\begin{equation*}\label{weighted-dirichlet}
    \Lambda_A \cap  \left[ {\begin{array}{cc}
   \psi(t)^\va  & 0 \\
   0 & t^\vb  \\
  \end{array} } \right]
  B_\nu (r_{\nu}) \ne \{{\bf 0}\}
\end{equation*}
for all sufficiently large $t$.
\end{defn}
For brevity, we write the set of $(\psi,\nu,\vw)$-Dirichlet matrices as $D_{\nu, \vw}(\psi)$. 
Note that the above property in general cannot be written in a way similar to \eqref{weightsystem}, with separate conditions involving the linear forms $A_i$ and the variables $q_j$. 
{For example, in the case where $m=n=1$, $\nu$ is the Euclidean norm on $\R^2$ and $\vw = (1,1)$ is the only possible choice for the weights, it is easy to see that $r_2= \left(\frac{4}{3}\right)^{1/4}$.
The corresponding condition for a real number $\alpha$ to be $(\psi,\nu,\vw)$-Dirichlet is that the inequality
\begin{equation*}
    \left(\frac{\alpha q -p}{\psi(t)}\right)^2 + \left(\frac{q}{t}\right)^2 < \frac{2}{\sqrt{3}}
\end{equation*}
has a solution in $(p,q)\in \Z\times \N$ for all sufficiently large $t$.}
 
It immediately follows from the definition of $r_\nu$ that $D_{\nu,\vw}(c\psi_{1})= \mr$ for any $c>1$. Also one can define
$$\DI_{\nu,\vw} :=\bigcup_{0<c<1}D_{\nu, \vw}(c\psi_1),$$ 
the set of \textit{weighted Dirichlet-improvable matrices with respect to $\nu$}, and use the same dynamical argument as in \cite[Theorem 1.4]{KW} to prove 

\begin{thm}\label{weighted-dirichlet-thm}
For any choice of a norm $\nu$ on $\R^d$ and a weight vector $\vw$, the set $\DI_{\nu,\vw}$ has Lebesgue measure zero.
\end{thm}

We are thus left with the following two problems:
\begin{itemize}
\item[1.] Find norms $\nu$ and weight vectors $\vw$ such that
\eq{full}{D_{\nu, \vw}(\psi_1)^c = \varnothing.}
\item[2.] Find  norms $\nu$ and weight vectors $\vw$ such that \eq{thick}{\DI_{\nu,\vw}\text{ is thick}.}
% has full \hd.
\end{itemize}

Both problems will be addressed in this paper for some specific choices of norms $\nu$, using a dynamical restatement of the property of being $(\psi,\nu,\vw)$-Dirichlet. 
The choice of norms in the theorems below arise from what is known or can be proved regarding the densest lattice-packings of their unit balls. This will be made abundantly clear in the proofs.

With regards to problem $1$ above, we have, like Theorem \ref{perfect-Dirichlet-sup}, a precise form of Dirichlet theorem in the following additional cases.
\begin{thm}\label{perfect-Dirichlet}
We have that $D_{\nu, \vw}(\psi_1) = \mr$
\begin{enumerate}
    \item[(a)] when $m=n=1$ and $\nu$ is any $\ell^p$ norm on $\R^{2}$;
    \item[(b)] when $m=2$, $n=1$, $\vw$ is arbitrary, and $\nu$ on $\R^3$ is of the form 
    \eq{cyl}{(x,y,z) \mapsto \max\big\{\eta(x,y), |z|\big\}\text{ for some norm }\eta\text{ on }\R^2.}
\end{enumerate}
\end{thm}
For Problem $2$, when $m=n=1$ and with only one possible choice of weights,
%there are no choices for weights, 
the thickness result was established in \cite[Theorem 1.3]{KR}.
For the {unweighted} case of the Euclidean norm in arbitrary dimensions it was established in \cite[Theorem 3.7]{KR}.
The result for the weighted supremum norm in arbitrary dimension follows from Theorem \ref{thickness-Dirichlet-sup}.
Presently we prove
\begin{thm}\label{thickness-Dirichlet}
The set $\mathbf{DI}_{\nu,\vw}$ is thick 
\begin{enumerate}
    \item[(a)] for any $m,n, \vw$, and when $\nu$ is the Euclidean norm on $\R^d$;
    \item[(b)] when $m=2$, $n=1$, $\vw$ is arbitrary, and $\nu$ on $\R^3$ is of the form 
    %$(x,y,z) \mapsto \max\{\mu(x,y), |z|\}$ 
    \equ{cyl}.
    % for some norm $\mu$ on $\R^2$.
\end{enumerate}
\end{thm}

%Specifically, it was shown in \cite{KR} that \equ{full} holds for $\nu$ being the Euclidean norm on $\R^2$; new examples can be found in Theorems \ref{perfect-Dirichlet-cylinder} and \ref{perfect-Dirichlet-finite}  below. 

%With regards to Question 2: it was shown in \cite{KR} that \equ{thick} holds for an arbitrary norm in $\R^2$, and also for the Euclidean norm on $\R^d$ in the unweighted case. Also, the paper \cite{KR1} studies the case 
Theorems \ref{perfect-Dirichlet} and \ref{thickness-Dirichlet} can be proved for certain other norms as well. See Proposition \ref{perfect-Dirchlet-finite} and Corollary \ref{thickness-finite} below for general   results applicable to other norms.

One might also ask whether or not the inclusion \eq{incl}{\BA_\vw\subset \DI_{{\nu,}\vw}} holds for some norms $\nu$ other than $\|\cdot\|_\infty$.  In Proposition \ref{topology-DS} we give a condition sufficient for  \equ{incl}, which in particular is valid for norms  of the form 
    %$(x,y,z) \mapsto \max\{\mu(x,y), |z|\}$ 
    \equ{cyl} as in Theorem \ref{thickness-Dirichlet}(b). However in general \equ{incl} is false: in fact for any $A\in \BA_\vw$ one can find a norm $\nu$ such that $A\notin \DI_{\nu,\vw}$. Moreover, the same holds for any \amr\ except for the case when $A$ is $\vw$-singular, or $A\in \Sing_{\vw}$. The latter set is defined as 
    $$
    \Sing_{\vw} := \bigcap_{0<c<1}D_{\nu, \vw}(c\psi_1).$$ 
    (%As Proposition \ref{dynamical-restatement} shows, 
    It is easy to see that the choice of the norm does not make a difference in this definition.)
   We prove
   \begin{thm}\label{sing}
For any   weight vector $\vw$,  
$$\Sing_{\vw} = \bigcap_{\nu \text{ a norm on }\R^d}\DI_{{\nu,}\vw}
$$
{In fact, for any fixed norm $\nu$ on $\R^d$, we have 
$$
\Sing_{\vw} = \bigcap_{g\in\SL_d(\R)}\DI_{{\nu\circ g,}\vw}.
$$}
\end{thm}
This characterization of singular systems of linear forms is new even in the unweighted case.

The structure of the paper is as follows; in the next section we give a dynamical interpretation of Dirichlet-improvability.
In particular, the relation to the \textit{critical locus} of a norm is clarified.
An effective equidistribution result on the space of lattices then yields the coarse form of Dirichlet's theorem as in Theorem \ref{weighted-dirichlet-thm}.
Theorems \ref{perfect-Dirichlet-sup}, \ref{thickness-Dirichlet-sup}, \ref{perfect-Dirichlet},    \ref{thickness-Dirichlet}(b) and \ref{sing} are proved in the next two sections by using the geometry of numbers to identify certain divergent subsets in the space of lattices.
Part (a) of Theorem \ref{thickness-Dirichlet} is proved in \S \ref{transversality} using results of the first-named author along with An and Guan.

\subsection*{Acknowledgements}
The authors are grateful to   Nikolay Moshchevitin  for helpful discussions, and to the anonymous referee for several useful comments.

\section{Dirichlet improvable matrices form a null set}
As before, $X_d$ denotes the space of unimodular lattices in $\R^d$, and $\nu$ stands for a norm on $\R^d$.
For any $r>0$ define
\begin{equation*}
    \mathcal{K}_\nu(r) := \big\{ \Lambda \in X_d : \Lambda \cap B_\nu\left(r\right) = \{{\bf 0}\}\big\}.
\end{equation*}
These sets are compact in view of Mahler's Compactness Criterion, and empty for %$r\leq 0$ and 
$r>r_\nu$, whereas for $0<r<r_\nu$, these give a system of neighborhoods of the non-empty compact \textit{critical locus} $\mathcal{L}_\nu:=\mathcal{K}_\nu(r_\nu)$.
Up to scaling, $\mathcal{L}_\nu$ gives the set of lattices witnessing the densest lattice-packings of the unit ball of $\nu$.
Further, given a weight vector as in \eqref{weights}, we have the following one-parameter subgroup of $\SL_d(\R)$:
\begin{equation}\label{a_s}
    a_s = \left[ {\begin{array}{cc}
   \left(e^{s}\right)^\va & 0 \\
   0 & \left(e^{-s}\right)^\vb \\
  \end{array} } \right].
\end{equation}
\begin{prop}\label{dynamical-restatement}
An $m \times n$ matrix $A$ belongs to $\mathbf{DI}_{\nu, \vw}$ if and only if there is some $0<r<r_\nu$ and $s_0>0$ such that
\begin{equation*}
 \{a_s\Lambda_A : s> s_0\} \cap \mathcal{K}_\nu(r) = \varnothing.
\end{equation*}
\end{prop}
\begin{proof}
Say $A \in \mathbf{DI}_{\nu, \vw}$, so that there is some $0<c<1$ with $A \in D_{\nu, \vw}(c\psi_{1})$.
The defining intersection condition for $D_{\nu,\vw}(c\psi_1)$ can be changed to
\begin{equation*}\label{asDirichlet}
a_s\Lambda_A \cap  a_s\left[ {\begin{array}{cc}
   \big(c\psi_{1}(t)\big)^\va & 0 \\
   0 & t^\vb \\
  \end{array} } \right]
  B_\nu\left(r_\nu\right) \ne \{{\bf 0}\}
\end{equation*}
for all sufficiently large $t$.
Putting
\begin{equation}\label{s-t}
s= \frac{1}{2}\ln\frac{t^2}{c},
\end{equation}
the condition becomes
\begin{equation*}\label{Dani-transformed1}
    a_s\Lambda_A \cap 
    \left[ {\begin{array}{cc}
   \left(\sqrt{c}\right)^\va & 0 \\
   0 & \left(\sqrt{c}\right)^\vb \\
  \end{array} } \right] B_\nu\left(r_\nu\right) \ne \{{\bf 0}\}
\end{equation*}
for all sufficiently large $s$.
Let
\begin{equation*}
r = r_\nu\cdot\max\left\lbrace {c}^{\alpha_1/2},\dots, {c}^{\alpha_m/2} , c^{\beta_1/2},\dots,  c^{\beta_n/2}\right\rbrace.
\end{equation*}
Since $c<1$, $r$ is less than $r_\nu$.
Thus, we have that 
\begin{equation}\label{asnotinKr}
    a_s\Lambda_A \notin \mathcal{K}_\nu(r)
\end{equation}
for all sufficiently large $s$.

Conversely, say we have a matrix $A$ for which there is an $0<r<r_\nu$ such that
\eqref{asnotinKr} holds
for all sufficiently large $s$.
Thus 
\begin{equation}\label{converseDyn}
    a_s\Lambda_A \cap B_\nu\left(r\right) \ne \{{\bf 0}\}
\end{equation}
for all sufficiently large $s$.
Condition \eqref{converseDyn} can be rewritten as
\begin{equation}\label{stot-intersection}
    \Lambda_A \cap \left[ {\begin{array}{cc}
   \left(e^{-s}\right)^\va & 0 \\
   0 & \left(e^{s}\right)^\vb \\
  \end{array} } \right]
  B_\nu\left(r\right) \ne \{{\bf 0}\}.
\end{equation}
So, if we define
\begin{equation*}
    c:= \left(\frac{r}{r_\nu}\right)^{\frac{2}{\gamma}} \text{ with } \gamma:= \max\{\beta_j\},
\end{equation*}
and define $t>0$ by the equation \eqref{s-t},
we see that 
\begin{equation*}\label{e^{-s/m}}
    e^{-s} = \frac{\sqrt{c}}{t} = \psi_1(t) \left(\frac{r}{r_\nu}\right)^{1/\gamma}\ \text{ and }\ e^{s} = \frac{t}{\sqrt{c}} = t \left(\frac{r_\nu}{r}\right)^{1/\gamma}.
\end{equation*}
From this we see that 
\begin{equation*}
    \frac{r}{r_\nu}e^{-s\alpha_i} = \left(\frac{r}{r_\nu}\right)^{1+\alpha_i/\gamma} \psi_{1}(t)^{\alpha_i}\ \text{ and } \ 
    \frac{r}{r_\nu} e^{s\beta_j} = \left(\frac{r}{r_\nu}\right)^{1-\frac{\beta_j}{\gamma}}t^{\beta_j}.
\end{equation*}
By choice of $\gamma$, we see that $\left(\frac{r}{r_\nu}\right)^{1-\frac{\beta_j}{\gamma}} \leq 1.$
Defining $c_1:=\left(\frac{r}{r_\nu}\right)^{\frac{1}{\alpha_i}+\frac{1}{\gamma}}$, which is less than $1$, condition \eqref{stot-intersection} then implies
\begin{equation*}
    \Lambda_A \cap 
    \left[ {\begin{array}{cc}
   \left(c_1\psi_{n/m}(t)\right)^\va & 0 \\
   0 & \left(t\right)^\vb \\
  \end{array} } \right]
  B_\nu\left(r_\nu\right) \ne \{{\bf 0}\}.
\end{equation*}
From this we can see that $A \in \mathbf{DI}_{\nu,\vw}$.
\end{proof}
Propositions of the above sort first appeared in \cite{d} and now go by the name `Dani's correspondence'.
\begin{cor}\label{perfect-Dirichlet-dynamics}
We have the equivalence 
\begin{equation*}
    A \notin D_{\nu,\vw}(\psi_{1}) \iff a_s\Lambda_A \in \mathcal{L}_\nu \text{ for an unbounded set of positive times } s.
\end{equation*}
\begin{proof}
It suffices to go through the above proof putting $c=1$ and $r=r_\nu$ in the forward and backward directions of the equivalence respectively.
\end{proof}
\end{cor}
In order to prove Theorem \ref{weighted-dirichlet-thm} we need the following equidistribution theorem of Kleinbock--Weiss \cite[Theorem 2.2]{KW}, see also \cite[Theorem 1.3]{KM4} for an effective version. The argument appears in \cite{KSY} in case of $\nu$ being the supremum norm and applies with little changes to the general case.
\begin{thm}\label{Kleinbock-Weiss-thm}
Let  $f \in C_c(X_d)$, $B\subset \mr$ be bounded with positive Lebesgue measure, and $\delta >0$ be given.
Then there exists an $s_0>0$ such that for all $s>s_0$,
\begin{equation*}\label{Kleinbock-Weiss-eqn}
\left|\frac{1}{\lambda(B)} \int_B f\left(a_s\Lambda_A\right)\,d\lambda(A) - \int_{X_d} f(x)\,d\mu\right| < \delta.
\end{equation*}
Here, the integrals are taken with respect to the Lebesgue measure $\lambda$ on $\mr$ and the Haar probability measure $\mu$ on $X_d$.
$\square$
\end{thm}
\begin{proof}[Proof of Theorem \ref{weighted-dirichlet-thm}]
We have $c<1$.
Let $r$ be associated to $c$ as in Proposition \ref{dynamical-restatement}.
We aim to show that for almost every $A \in \mr$, there is an unbounded positive sequence $(s_k)$ such that
\begin{equation}\label{inKr}
    a_{s_k}\Lambda_A \in \mathcal{K}_\nu(r).
\end{equation}
This and Proposition \ref{dynamical-restatement} then show that almost every $A \notin D_{\nu,\vw}(c\psi_{1})$. 
For $i\in \mathbb{N}$, if the set
\begin{equation*}
    B_i := \bigcap_{s>i}\left\lbrace A \in \mr :  a_s\Lambda_A \notin \mathcal{K}_\nu(r)\right\rbrace
\end{equation*}
has positive Lebesgue measure, choose $B\subset B_i$ compact with positive measure as well.
Take a non-negative $f \in C_c(X_d)$ %non-negative and 
which is supported on $\mathcal{K}_\nu(r)$, and choose $\delta = \frac{1}{2}\int_{X_d} f\,d\mu$.
Applying Theorem \ref{Kleinbock-Weiss-thm} with $s>i$, we get a contradiction.
Thus each $B_i$ has measure zero and thus so does their union.
Hence we have shown that Lebesgue almost every $A \in \mr$ has an unbounded positive sequence $(s_k)$ for which \eqref{inKr} holds.
\end{proof}

\section{Dirichlet's theorem via divergence}

For the rest of the paper we fix a weight vector as in \eqref{weights} and the one-parameter subgroup $\{a_s\}$ of $\SL_d(\R)$ as in 
\eqref{a_s}.  We now address Problem 1  regarding Dirichlet's theorem in the form \equ{full}.
First, a general condition implying the result.
\begin{prop}\label{perfect-Dirichlet-generalthm}
Say $\nu$ is a norm in $\R^d$ with $\mathcal{L}_\nu = \bigcup \mathcal{Z}_i$ a finite union of compact subsets such that each $\mathcal{Z}_i$ has either one of the following properties.
\begin{enumerate}
    \item[(i)] For every $\Lambda\in \mathcal{Z}_i$ and compact $\mathcal{K} \subset X_d$, there is a $t_0$ such that for all $s>t_0$, $a_s\Lambda \notin \mathcal{K}$. That is, every $\Lambda \in \mathcal{Z}_i$ is forward divergent.
    \item[(ii)] For every $\Lambda \in \mathcal{Z}_i$ and compact $\mathcal{K} \subset X_d$, there is a $t_0$ such that for all $s<t_0$, $a_s\Lambda \notin \mathcal{K}$. That is,  every $\Lambda \in \mathcal{Z}_i$ is backward divergent.
\end{enumerate}
Then $D_{\nu,\vw}(\psi_1) = \mr$.
\end{prop}
\begin{proof}
For the sake of contradiction, say that $A \notin D_{\nu,\vw}(\psi_1)$.
By Corollary \ref{perfect-Dirichlet-dynamics}, there is an unbounded positive sequence $(s_k)$ such that
for each $k$, $a_{s_k}\Lambda_A \in \mathcal{L}_\nu$.
By the above {finiteness} hypothesis we might as well assume $\mathcal{L}_\nu$ itself has one of the properties (i) or (ii).
Observe that compactness implies that there is a uniform $t_0$ in the above conditions which works for every $\Lambda \in \mathcal{L}_\nu$.
We now separate into two cases. 
\begin{enumerate}
\item[(i)] We can find $t_0$ such that for all $s>t_0$, \begin{equation}\label{locus-moves-out}
a_s\mathcal{L}_\nu \cap \mathcal{L}_\nu = \varnothing.
\end{equation}
This contradicts the fact that for every $k$, $a_{s_k}\Lambda_A (= a_{s_k-s_1}a_{s_1}\Lambda_A)$ belongs to $\mathcal{L}_\nu$.
\item[(ii)] Find $t_0$ such that for all $s<t_0$, 
\eqref{locus-moves-out} holds.
This contradicts the fact that for every $k$, $a_{s_1}\Lambda_A (= a_{s_1-s_k}a_{s_k}\Lambda_A) \in \mathcal{L}_\nu$.
\end{enumerate}
 Thus $D_{\nu,\vw}(\psi_1) = \mr$.
\end{proof}

\begin{proof}[Proof of Theorem \ref{perfect-Dirichlet-sup}]
Let $B$ denote the set of upper triangular unipotent $d\times d$ matrices.
It is a well-known theorem of Haj\'os \cite{H} that the set $\mathcal{L}_\infty$ is exactly the union
\begin{equation}\label{hajos}
    \bigcup \left\lbrace w B w\SL_d(\Z) : w \text{ is a permutation matrix}\right\rbrace.
\end{equation}
{From this we get that for every permutation matrix $w$, there is some fixed standard basis vector $\ve_i$ which belongs to every $\Lambda \in wBw\SL_d(\Z)$.
From the description of $a_s$ in \eqref{a_s}, we see that, according to whether $m<i$ or $i\leq m$, $\ve_i$ is contracted by $a_s$ either for $s>0$ or $s<0$.
Thus,  for each permutation matrix $w$, we are in one of the two situations of Proposition \ref{perfect-Dirichlet-generalthm}.}
\end{proof}
\begin{proof}[Proof of Theorem \ref{perfect-Dirichlet}(b)]
\cite[Proposition 5.1]{kr} asserts that whenever $\nu$ is a cylindrical norm on $\R^3$ as in \equ{cyl}, the critical locus in $X_3$ is contained in the union of
\begin{equation}\label{cylinderical-locus}
  \mathcal{Z}_1 := \left\lbrace
        \left[ {\begin{array}{ccc}
   * & * & 0\\
   * & * & 0 \\
   * & * & *
  \end{array} } \right]\Z^3
     \right\rbrace
\text{ and }
%\begin{equation}\label{zminus}
  \mathcal{Z}_2 := \left\lbrace
        \left[ {\begin{array}{ccc}
   * & * & * \\
   * & * & * \\
   0 & * & *
  \end{array} } \right]\Z^3
     \right\rbrace.
     \end{equation}
Moreover, since we have $m=2$ and $n=1$ by hypothesis,
\begin{equation}\label{m=2n=1}
    a_s = \left[ {\begin{array}{ccc}
   e^{s\alpha_1} & 0 & 0\\
   0 & e^{s\alpha_2} & 0 \\
   0 & 0 & e^{-s}
  \end{array} } \right].
\end{equation}
Thus, if $\Lambda \in \mathcal{Z}_1$, it contains a vector contracted by $a_s$ for $s>0$. And if $\Lambda \in \mathcal{Z}_2$, it contains a vector contracted by $a_s$ for $s<0$.
Applying Proposition \ref{perfect-Dirichlet-generalthm}, we are done.
\end{proof}
We also have the following simple but useful result:
{\begin{prop}\label{perfect-Dirchlet-finite}
Let $\nu$ be a norm on $\R^d$ such that the critical locus $\mathcal{L}_\nu$ is finite.
Then $D_{\nu,\omega}(\psi_1) = \mr$.
\end{prop}
\begin{proof}
Again, by Proposition \ref{dynamical-restatement}, any $A \notin D_{\nu,\omega}$ would give rise to a periodic orbit $\{a_s\Lambda_A\}$.
On the other hand, $\Lambda_A$ is backward divergent under the flow $a_s$.
\end{proof}}
\begin{proof}[Proof of Theorem \ref{perfect-Dirichlet}(a)]
This has already been proved for $p=2$ in \cite[Theorem 1.4]{KR}, and for $p=\infty$ in Theorem \ref{perfect-Dirichlet-sup}.
For the other cases, the work \cite{GGM} shows that $\mathcal{L}_p$ is finite.
Thus we are done by applying Proposition \ref{perfect-Dirchlet-finite}.
\end{proof}
\begin{rem}
Other examples of norms which are known to have finite critical locus are norms in $\R^2$ induced by hexagons, as well as the $\ell^1$ norm in $\R^3$. 
For the former fact see \cite[\S V.8.4, Lemma 13]{Ca} and for the latter see \cite{M} or the discussion in the pages prior to \cite[Equation (4), page 346]{GL}.
\end{rem}

\section{Thickness results via divergence}\label{sup-thick-section}

Some similar observations about divergence in the space of lattices lead us to solutions of Problem 2 as well.
Recall the set $\mathbf{BA}_\vw$ of $\vw$-badly approximable matrices defined in \eqref{wBA}.
It is well known (see \cite[Theorem 2.5]{K98}) that
\begin{equation*}
    A \in \mathbf{BA}_\vw \iff \{a_s\Lambda : s>0\} \text{ is bounded in } X_d.
\end{equation*}
%See \cite[Theorem 2.5]{K98} for further details.
We now give a general proposition giving sufficient conditions (on the norm $\nu$) which ensure that $\mathbf{BA}_\vw$ is a subset of $\mathbf{DI}_\nu$.
%\comm{removed earlier comment about Davenport-Schmidt.}
\begin{prop}\label{topology-DS}
If $\nu$ is a norm on $\R^d$ such that every $\Lambda \in \mathcal{L}_\nu$ has the property that
\begin{equation*}
    \left\lbrace a_s\Lambda : s \in \R \right\rbrace \text{ is unbounded in } X_d,
\end{equation*}
then $\mathbf{BA}_\vw$ is contained in $\mathbf{DI}_{\nu,\vw}$.
\end{prop}
\begin{rem}
To be precise, being unbounded means that for each compact $\mathcal{K}\subset X_d$, there is some $s \in \R$ such that $a_s\Lambda \notin \mathcal{K}$.
\end{rem}
\begin{proof}
We again use the characterization in Proposition \ref{dynamical-restatement}.
Say $\nu$ is a norm with the property as above.
Say $A \in \mathbf{BA}_\vw$.
Say further, contrary to the theorem, that there is an unbounded positive sequence $(s_k)$ and a lattice $\Lambda \in \mathcal{L}_\nu$ such that
$
    a_{s_k}\Lambda_A \to \Lambda.
$
Let $\mathcal{K}\subset X_d$ be a compact set such that 
\begin{equation*}
    \{a_s\Lambda_A : s>0\} \subset \mathcal{K}.
\end{equation*}
We consider two cases.
\begin{enumerate}
    \item[(i)] $\{a_s\Lambda : s>0\}$ is unbounded. This implies that there is a positive time $t$ for which
$
 a_{t}\Lambda \notin \mathcal{K}.
$
Let $\mathcal{V}$ be a neighborhood of $\Lambda$ such that 
\begin{equation}\label{nbd-notinK}
    a_{t}\mathcal{V} \subset X_d \smallsetminus \mathcal{K}.
\end{equation}
Thus, for large enough $k$, we have 
$
    a_{t+s_k}\Lambda_A \notin \mathcal{K},
$
a contradiction.
\item[(ii)] For the second case, we assume that
$
    \{a_s\Lambda : s<0\} \text{ is unbounded.}
$
This means we have a negative $t$ for which $a_t\Lambda \notin \mathcal{K}$.
Let $\mathcal{V}$ again be a neighborhood such that \eqref{nbd-notinK} holds.
We have that for large $k$, $a_{t+s_k}\Lambda_A \notin \mathcal{K}$.
On observing that $t+s_k$ is positive for large $k$, we have a contradiction.
\end{enumerate}
Thus, any $A$ in ${\BA_\vw}$ must belong to $\mathbf{DI}_{\nu,\vw}$.
\end{proof}
\begin{proof}[Proof of Theorem \ref{thickness-Dirichlet-sup}]
As was observed before, it follows from the expression \eqref{hajos}  for the critical locus
%, we see that 
$\mathcal{L}_\infty$ that 
every $\Lambda \in \mathcal{L}_\infty$ contains one of the basis vectors $\ve_i$.
% is exactly the union
%\begin{equation}
%    \bigcup \left\lbrace w B w\SL_d(\Z) : w \text{ is a permutation matrix}\right\rbrace.
%\end{equation}
{%Thus, as was observed before, every $\Lambda \in \mathcal{L}_\infty$ contains one of the basis vectors $\ve_i$.
So, according to whether $m<i$ or $i\leq m$, $\ve_i$ is contracted by $a_s$ either for $s>0$ or $s<0$.
This, of course, implies that $\{a_s\Lambda\}$ is unbounded and we can apply Proposition \ref{topology-DS}.}
\end{proof}
\begin{proof}[Proof of Theorem \ref{thickness-Dirichlet}(b).]
{Again, from \eqref{cylinderical-locus} and \eqref{m=2n=1} describing the critical locus and the flow respectively, we see that each $\Lambda \in \mathcal{L}_\nu$ is either forward or backward divergent (hence also unbounded) with respect to $a_s$.
Thus Proposition \ref{topology-DS} applies.} 
\end{proof}
Perhaps now is a good time to %remind ourselves 
observe that the conclusion of Proposition \ref{topology-DS} does not always hold. More precisely, for any $A\notin  \Sing_\vw$ there exists a norm $\nu$ on $\R^d$ such that $A$ does not belong to $\DI_\vw$.
\begin{proof}[Proof of Theorem \ref{sing}]
%By Proposition \ref{dynamical-restatement} we see 
It is well known (see \cite[Theorem 7.4]{K98}, or  \cite[Proposition 2.12]{d} for a version with equal weights) that $A\in \Sing_\vw$ if and only if $\Lambda_A$ is forward divergent under $a_s$.
And by divergence, any such element must avoid any given critical locus after a certain time.
Thus $\Sing_\vw$ is contained in each of the intersections in the theorem.

{
To complete the proof, it now suffices to show that, for a fixed norm $\nu$,
\begin{equation*}
\bigcap_{g \in \SL_d(\R)} \mathbf{DI}_{\nu\circ g, \vw} \subset \Sing_\vw.
\end{equation*}
Take $A \in \mr$ that is Dirichlet-improvable for all norms of the form $\nu\circ g$.
In order to show that $A$ is singular, it suffices to show that for every   $\Lambda \in X_d$, there is a neighborhood $\mathcal{V}$ of $\Lambda$ and some time $s_0$ such that the orbit $\{a_s\Lambda_A : s>s_0 \}$ avoids $\mathcal{V}$.

Fix $\Lambda \in X_d$ and pick some $g \in \SL_d(\R)$ such that $g\Lambda \in \mathcal{L}_\nu$. Since $g^{-1}\mathcal{L}_\nu = \mathcal{L}_{\nu\circ g}$, we see that $\Lambda \in \mathcal{L}_{\nu \circ g}$. By Dirichlet-improvability  of $A$ with respect to $\nu\circ g$, we see that there is an $r < r_{\nu\circ g}$ and some $s_0$ such that 
$$
a_s\Lambda_A \notin \mathcal{K}_{\nu \circ g}(r) \text{ for all } s>s_0.
$$
As observed before, $\mathcal{K}_{\nu \circ g}(r)$ for $r< r_{\nu \circ g}$ is an open neighborhood of $\mathcal{L}_{\nu \circ g}$, and so we are done.
}
\end{proof}

\section{Thickness results via transversality}\label{transversality}

In order to prove the thickness result for the Euclidean norm, we use a result of the first-named author with An and Guan \cite{AGK}.
They give a very general condition on the critical locus $\mathcal{L}_\nu$ which guarantees
that the set of \amr\ such that the trajectory \linebreak  $\{a_shx : s>0\}$ eventually stays away from $\mathcal{L}_\nu$ is winning in the sense of Schmidt.
%\begin{rem}
More precisely, the results in  \cite{AGK} deal with a modified version of Schmidt's winning property called 
\textit{hyperplane absolute winning} (HAW).
For the definition of the HAW property, see \cite[\S 2]{BFKRW} or \cite[\S 2.1]{AGK}. 
HAW implies winning in the sense of Schmidt \cite{S}, and this in turn
implies thickness. Furthermore, the class of HAW sets, like those which are winning, is closed under
countable intersections.
%\end{rem}

To state the aforementioned condition we need some notation.
Let $G$ denote $\SL_d(\R)$, and let $\mathfrak{g}$ denote its Lie algebra $\mathfrak{sl}_d(\R)$.
Let $H \subset G$ denote the subgroup
$$
H=\{u_A: A\in \mr\},
$$
and let $\mathfrak{h}$ denote its Lie algebra.
Fixing weights $\vw = (\va,\vb)$, let $F \subset G$ denote the subgroup
\begin{equation*}
    F = \left\lbrace a_s : s\in \R \right\rbrace
\end{equation*}
where $a_s$ is as in %equation 
\eqref{a_s}.
Let $D \in \mathfrak{g}$ denote the the diagonal element
$$
D = 
  \left[ {\begin{array}{cc}
   \left(1\right)^\va & 0 \\
   0 & -\left(1\right)^\vb \\
  \end{array} } \right] = \operatorname{diag}(\alpha_1,\dots,\alpha_m, -\beta_1,\dots,-\beta_n)
$$
so that 
$$
a_s = \exp(sD).
$$

The adjoint action $\operatorname{ad}(D):\mathfrak{g} \to \mathfrak{g}$ is diagonable:
If we let $E^{i,j}$ denote the $d \times d$ matrix with $1$ in the $(i,j)$-entry and $0$ everywhere else we see that,
when $i\neq j$,
\begin{equation*}
    \operatorname{ad}(D)E^{i,j} = \left(D_{ii} - D_{jj}\right)E^{i,j},
\end{equation*}
and that
\begin{equation*}
    \operatorname{ad}(D)\left(E^{i,i} - E^{j,j}\right) = 0.
\end{equation*}
Here $D_{ij}$ denotes the $(i,j)$-entry of $D$.
Thus, if we let $\lambda$ run over the eigenvalues of $\operatorname{ad}(D)$, we have an eigenspace decomposition
\begin{equation*}
    \mathfrak{g} = \bigoplus \mathfrak{g}_\lambda.
\end{equation*}
Let $\rho$ denote the largest eigenvalue, and let $q:\mathfrak{g}\to \mathfrak{g}$ be the projection 
%on $\mathfrak{g}$ 
with image and kernel
\begin{equation*}
    \bigoplus_{\lambda=\rho} \mathfrak{g}_\lambda \text{ and } \bigoplus_{\lambda < \rho} \mathfrak{g}_\lambda
\end{equation*}
respectively.
Let $\mathfrak{h}^{max}$ denote the image $q(\mathfrak{h})$, and let $H^{max}$ denote the connected subgroup of $G$ generated by $\mathfrak{h}^{max}$.
\begin{rem}\label{hmax}
Note that, from the definition of $D$, the eigenvectors with maximal eigenvalues must occur as matrices $E^{i,j}$ with $i \leq m$ and $j \geq n$.
Thus $\mathfrak{h}^{max}$ (which is a subset of $\mathfrak{h}$) is never the zero subspace.
\end{rem}
We also use the notation $T_x(M)$ to denote the tangent space of a submanifold $M$ of $X_d$ at a point $x\in X_d$.
\begin{defn}
A compact submanifold $Z \subset X_d$ is said to be $(F, H^{max})$-transversal if for all $z \in Z$, 
\begin{enumerate}
    \item[(i)] $T_z(Fz) \not\subset T_z(Z)$;
    \item[(ii)] $T_z(H^{max}z) \not\subset T_z(Z) \oplus T_z(Fz)$.
\end{enumerate}
\end{defn}
We can finally state the relevant result from \cite[Theorem 2.8]{AGK}.
\begin{thm}\label{AGKtheorem}
Keeping with the notation above, if $Z \subset X_d$ is an $(F, H^{max})$-transversal compact submanifold, then for any $x \in X_d$,
\begin{equation*}
    \left\lbrace h\in H :  \overline{\{a_shx : s>0\}} \cap Z = \varnothing \right\rbrace
\end{equation*}
is {HAW} in $H$.
\end{thm}
{Clearly zero-dimensional submanifolds are $(F,H^{max})$-transversal.
And since countable intersections of winning sets are winning, on applying the above theorem to the case where $x\in X_d$ is the standard lattice, we have
\begin{cor}\label{thickness-finite}
If $\nu$ is a norm on $\R^d$ such that $\mathcal{L}_\nu$ is finite, then $\mathbf{DI}_{\nu,\omega}$ is thick. $\square$
\end{cor}}
We can also apply Theorem \ref{AGKtheorem} to get
\begin{proof}[Proof of Theorem \ref{thickness-Dirichlet}(a).]
Recall from Proposition \ref{dynamical-restatement} that $A \in \mathbf{DI}_{2,\vw}$ if and only if there is some $r<r_\nu$ such that
\begin{equation*}
    a_s\Lambda_A \notin \mathcal{K}_2(r)
\end{equation*}
for all sufficiently large $s$.
Here we are considering a neighborhood of the compact set $\mathcal{L}_2 \subset X_d$ which is a finite union of $\operatorname{SO}(d)$-orbits (see \cite[Theorem 3.7]{KR}).
The Lie algebra $\mathfrak{so}(d)$ consists of skew-symmetric matrices and it then becomes straightforward to check that each $\operatorname{SO}(d)$-orbit is an $(F,H^{max})$-transversal submanifold.
Indeed, after identifying with $\mathfrak{g}$, we see that $T_z(Fz) = \text{span}_\R\{ D\}$, while $T_z(H^{max}z)$ includes nonzero upper triangular matrices, so that it is not contained in $\mathfrak{so}(d)\oplus \text{span}_\R\{ D\}$.
Thus, Theorem \ref{AGKtheorem} shows that $\mathbf{DI}_{2,\vw}$ contains a finite intersection of winning sets, and thus is itself thick.
\end{proof}

%\comm{Finally, somewhere we should mention that we don't know whether or not \equ{full}, \equ{thick} and the conclusion of Proposition \ref{topology-DS} hold in general. Or are there any counterexamples to the latter? 
%\bf Can you give me more time to write down the ideas for $\BA \not\subset \mathbf{DI}_2$
%If they turn out to be interesting, we can maybe write a paper with Ayreena.}


\begin{thebibliography}{BFKRW\ }
  \small
  
  \bibitem[AD]{AD} N.\ Andersen and W.\ Duke, \textsl{On a theorem of Davenport and Schmidt}, Acta Arith. {\bf 198} (2021), no.\ 1,  37--75.
  
   \bibitem[AGK]{AGK} J.\ An, L.\ Guan and D.\ Kleinbock, \textsl{Nondense orbits on homogeneous spaces and applications to geometry and number theory}, %preprint (2020), {\tt https://arxiv.org/pdf/2001.05174.pdf}, 
   % to appear in 
    Ergodic Theory Dynam. Systems (2021), DOI: {\tt https://doi.org/10.1017/etds.2021.4}.   
    
    \bibitem[BFKRW]{BFKRW} R.\ Broderick, L.\ Fishman, D.\ Kleinbock, A.\ Reich and B.\ Weiss, \textsl{The set of badly approximable vectors is strongly ${\mathcal C}^1$ incompressible}, Math.\ Proc.\ Cambridge Philos.\ Soc.\ {\bf 153} (2012), no.\ 2, 319--339.
    
   \bibitem[C1]{Cassels} J.\,W.\,S.\ Cassels, \textsl{An Introduction to Diophantine Approximation}, Cambridge Tracts in Mathematics and Physics, Cambridge University Press, London, 1957.
   
    \bibitem[C2]{Ca} \bysame,
     \textsl{An introduction to the geometry of numbers},
 {Die Grundlehren der mathematischen Wissenschaften, Bd.\ 99 Springer-Verlag},
      {1959},
     {viii+344}.
 
   \bibitem[Da]{d} S.\,G.\ Dani, \textsl{Divergent trajectories of flows on homogeneous spaces and Diophantine approximation}, J.\ Reine Angew.\ Math. {\bf 359} (1985), 55–89.

  \bibitem[DS]{Davenport-Schmidt}  H.\ Davenport and W.\,M.\ Schmidt,
\textsl{Dirichlet's theorem on diophantine approximation}, in: Symposia
Mathematica, Vol.\ IV (INDAM, Rome, 1968/69),  1970.
  
 % \bibitem[DK]{DK} M.\ Dodson and S.\ Kristensen, \textsl{{Hausdorff dimension and {D}iophantine approximation}}, \textsl{Fractal geometry and applications: a jubilee of {B}eno\^{\i}t
  %            {M}andelbrot. {P}art 1}, {Proc. Sympos. Pure Math.}, {\bf 72} (2004), 
  %            {305--347}.
  
  \bibitem[H]{H}
      {G.\ Haj\'{o}s},
     \textsl{\"{U}ber einfache und mehrfache {B}edeckung des
              {$n$}-dimensionalen {R}aumes mit einem {W}\"{u}rfelgitter},
   {Math. Z.} {\bf 47} ({1941}), {427--467}.
   
   \bibitem[GGM]{GGM} N.\,M.\ Glazunov,  A.\,S.\ Golovanov and A.\,V.\  Malyshev, \textsl{Proof of the {M}inkowski conjecture on the critical
              determinant of the region {$|x|^p+|y|^p<1$}}, {Zap.\ Nauchn.\ Sem.\ Leningrad.\ Otdel.\ Mat.\ Inst.\ Steklov
              (LOMI)} {\bf 151} (1986), 40--53.
              
\bibitem[GL]{GL} P.\ Gruber and C.\ Lekkerkerker, \textsl{Geometry of numbers}, North-Holland Mathematical Library, {\bf 37}, North-Holland Publishing Co., Amsterdam, 1987.

  \bibitem[Kl]{K98} D.\  Kleinbock,
\textsl{Flows on homogeneous spaces and {D}iophantine properties of
              matrices},
{Duke Math. J.} {\bf 95} (1998), 107--124.


\bibitem[KM]{KM4} D.\ Kleinbock and G.\,A.\ Margulis, \textsl{On effective equidistribution of expanding translates of certain orbits in the space of
lattices}, in: Number Theory, Analysis and Geometry, Springer, New York, 2012, pp. 385--396.

\bibitem[KSY]{KSY} D.\ Kleinbock, A.\ Str\"ombergsson, and S.\ Yu, \textsl{A measure estimate in geometry of numbers and improvements to Dirichlet's theorem}, arXiv preprint {\tt https://arxiv.org/abs/2108.04638} (2021).

\bibitem[KR1]{KR} D.\ Kleinbock  and  A.\ Rao, \textsl{A zero-one law for uniform Diophantine approximation in Euclidean norm}, arXiv preprint {\tt arXiv:1910.00126} (2019), to appear in Internat.\ Math.\ Res.\ Notices.

\bibitem[KR2]{kr} \bysame, \textsl{Abundance of Dirichlet-improvable pairs with respect to arbitrary norms}, 
arXiv preprint {\tt https://arxiv.org/abs/2107.10298} (2021), to appear in Mosc.\ J.\ Comb.\ Number Theory.


\bibitem[KWa1]{kw1} D.\ Kleinbock  and N.\ Wadleigh,  \textsl{A zero-one law for improvements to Dirichlet's Theorem}, Proc.\ Amer.\ Math.\ Soc.\ {\bf 146} (2018), no.\ 5, 1833--1844.


\bibitem[KWa2]{kw} \bysame,    \textsl{An inhomogeneous Dirichlet theorem via shrinking targets}, Compos. Math.,
{\bf 155} (2019), no.\ 7, 1402--1423.

\bibitem[KWe1]{KW} D.\ Kleinbock and B.\ Weiss, 
\textsl{Dirichlet's theorem on {D}iophantine approximation and homogeneous flows}, {Journal of Modern Dynamics} {\bf 2} (2008), 43--62. 

\bibitem[KWe2]{KW2} \bysame, \textsl{Modified Schmidt games and Diophantine approximation with weights},
Advances in Mathematics {\bf 223} (2010),  1276--1298.

 \bibitem[M]{M} H.\ Minkowski, \textsl{Dichteste  gitterf\"ormige  Lagerung  kongruenter  K\"orper},  Nachr.  K. 
Ges.  Wiss.  G\"ottingen (1904),  311-355.  (Reprinted in Gesammelte Abhandlungen  II, 3-42)

\bibitem[PV]{PV} A.\ Pollington and S.\ Velani, \textsl{On simultaneously badly approximable numbers}, J. London Math. Soc. (2) {\bf 66} (2002), no. 1, 29–40.

\bibitem[Sc1]{S} W.\,M.\ Schmidt,    \textit{On badly approximable numbers and certain games}, Trans.\ Amer.\ Math.\ Soc.\ {\bf 123} (1966), 178--199.

\bibitem[Sc2]{Schmidt-BA} \bysame,    \textsl{Badly approximable systems of
linear forms}, J.\ Number Theory {\bf 1} (1969), 139--154.

\bibitem[Su]{Suess} F.\ S\"uess, \textsl{Simultaneous Diophantine approximation on affine subspaces and Dirichlet improvability}, arXiv preprint {\tt https://arxiv.org/abs/1711.08288} (2017).

%\bibitem[Z]{Zong} C.\ Zong, \textsl{On the packing densities and the covering densities of the         {C}artesian products of convex bodies}, {Monatshefte f\"{u}r Mathematik} {\bf 145} (2004), {73--81}.

    \end{thebibliography}
\end{document}